\documentclass{amsart}

\usepackage[a4paper,hmargin=3.45cm,vmargin=4cm]{geometry}
\usepackage{amsfonts,amssymb,amscd,amstext}
\usepackage{mathtools}
\mathtoolsset{showonlyrefs=true} 
\usepackage[colorlinks=true,linkcolor=blue,citecolor=red]{hyperref}

\renewcommand{\Bbb}{\mathbb}
\renewcommand{\leq}{\leqslant}
\renewcommand{\geq}{\geqslant}
\newcommand{\la}{\lambda}
\newcommand{\ptl}{\partial}
\newcommand{\rr}{{\mathbb{R}}}
\newcommand{\h}{\mathcal{H}}
\newcommand{\sub}{\subset}
\newcommand{\subeq}{\subseteq}
\newcommand{\escpr}[1]{\big<#1\big>}
\newcommand{\Sg}{\Sigma}
\newcommand{\sg}{\sigma} 

\newcommand{\Om}{\Omega}
\newcommand{\eps}{\varepsilon}
\newcommand{\var}{\varphi}
\newcommand{\ga}{\gamma}
\newcommand{\mnh}{|N_{h}|}
\newcommand{\nuh}{\nu_{h}}
\newcommand{\stres}{\mathbb{S}^3}
\newcommand{\s}{\mathcal{S}}
\newcommand{\be}{\mathcal{B}}
\newcommand{\sph}{\mathbb{S}}

\newcommand{\ttt}{\mathcal{T}}

\DeclareMathOperator{\divv}{div}

\newtheorem{theorem}{Theorem}[section]
\newtheorem{proposition}[theorem]{Proposition}

\newtheorem{corollary}[theorem]{Corollary}

\theoremstyle{definition}

\newtheorem{remarks}[theorem]{Remarks}

\theoremstyle{remark}

\numberwithin{equation}{section}

\begin{document}

\bibliographystyle{amsplain}

\title[Area-minimizing properties of Pansu spheres in $\stres$]
{Area-minimizing properties of Pansu spheres \\ in the sub-riemannian $3$-sphere}

\author[Ana Hurtado]{Ana Hurtado}
\address{Departamento de Geometr\'{\i}a y Topolog\'{\i}a and Excellence Research Unit ``Modeling Nature'' (MNat), Universidad de Granada, E-18071,
Spain.}
\email{ahurtado@ugr.es}

\author[C\'esar Rosales]{C\'esar Rosales}
\address{Departamento de Geometr\'{\i}a y Topolog\'{\i}a and Excellence Research Unit ``Modeling Nature'' (MNat) Universidad de Granada, E-18071,
Spain.} 
\email{crosales@ugr.es}

\date{\today}

\thanks{The authors were supported by MINECO grant 
MTM2017-84851-C2-1-P and Junta de Andaluc\'ia grants A-FQM-441-UGR18 and FQM325.} 



\begin{abstract}
We consider the sub-Riemannian $3$-sphere $(\stres,g_h)$ obtained by restriction of the Riemannian metric of constant curvature $1$ to the planar distribution orthogonal to the vertical Hopf vector field. It was shown in \cite{hr1} that $(\stres,g_h)$ contains a family of spherical surfaces $\{\s_\la\}_{\la\geq 0}$ with constant mean curvature $\lambda$. In this work we first prove that the two closed half-spheres of $\s_0$ with boundary $C_0=\{0\}\times\sph^1$ minimize the sub-Riemannian area among compact $C^1$ surfaces with the same boundary. We also see that the only $C^2$ solutions to this Plateau problem are vertical translations of such half-spheres.  Second, we establish that the closed $3$-ball enclosed by a sphere $\s_\la$ with $\la>0$ uniquely solves the isoperimetric problem in $(\stres,g_h)$ for $C^1$ sets inside a vertical solid tube and containing a horizontal section of the tube. The proofs mainly rely on calibration arguments.
\end{abstract}

\maketitle

\thispagestyle{empty}

\section{Introduction}
\label{sec:intro}
\setcounter{equation}{0} 

The study of variational questions related to the area in sub-Riemannian spaces has been focus of attention in the last years, especially in the Heisenberg group, which is the simplest non-trivial flat sub-Riemannian manifold. Our aim in this work is to investigate area-minimizing surfaces in the most relevant $3$-dimensional sub-Riemannian model of positive curvature, namely the $3$-sphere $(\stres,g_h)$. Here the sub-Riemannian structure comes from the restriction of the Riemannian metric of constant curvature $1$ to the \emph{horizontal} planar distribution orthogonal to the fibers of the Hopf fibration, see Sections~\ref{subsec:sphere} and \ref{subsec:hopf}.  In this setting we will analyze the \emph{Plateau problem} and the \emph{isoperimetric problem}, where we seek minimizers of the area with a boundary or volume constraint, respectively. More precisely, we will show that the \emph{Pansu spherical surfaces} in $(\stres,g_h)$ provide solutions to these problems under certain conditions.

In the first Heisenberg group, Pansu~\cite{pansu2} employed a Santal\'o formula to derive an isoperimetric inequality. He also explained in \cite{pansu1} how to construct constant mean curvature spheres with rotational symmetry around the center of the group, and conjectured that the isoperimetric regions are, up to congruence, the topological $3$-balls enclosed by these spheres. This conjecture has been supported by several statements where additional geometric and regularity conditions are assumed, see \cite[Ch.~8]{survey} and the introductions in \cite{ritore-calibrations,hr4} for more details and references. In the sub-Riemannian sphere $(\stres,g_h)$ an isoperimetric inequality was proved by Chanillo and Yang~\cite{chanillo-yang} by means of a sub-Riemannian Santal\'o formula, see also the paper of Prandi, Rizzi and Seri~\cite{santalo-sr}. In this context the existence of isoperimetric regions is guaranteed by compactness, see for instance \cite{miranda-bv} and \cite[Ch.~5]{survey}. In \cite{hr1} the authors discovered a one-parameter family $\{\s_\la\}_{\la\geq 0}$ of spherical surfaces of revolution having constant mean curvature $\la$ in $(\stres,g_h)$. In later works we established some results concerning the spherical Pansu conjecture~\cite[Sect.~6]{hr2} saying that, up to congruence, any isoperimetric region in $(\stres,g_h)$ is bounded by some sphere $\s_\la$. Let us describe these results. 

In \cite[Sect.~4]{hr2} we derived an Alexandrov-type theorem stating that the unique compact $C^2$ critical points for the isoperimetric problem in $(\stres,g_h)$, and contained within an open hemisphere, are congruent to a sphere $\s_\la$. This was obtained by using the first variation formulas for sub-Riemannian area and volume~\cite{hr1}, the ruling property of $C^2$ surfaces with constant mean curvature~\cite{chmy,hr1} 
and the local behaviour of such surfaces around the singular set -where the tangent plane is horizontal- studied by Cheng, Hwang, Malchiodi and Yang~\cite[Sect.~3]{chmy}. This theorem does not hold if the surface is not contained inside an open hemisphere since it is possible to find critical points for the isoperimetric problem in $(\stres,g_h)$ with the topology of the $2$-torus $\sph^1\times\sph^1$, see \cite{chmy,hr1,hr3}. This phenomenon is coherent with the topological restriction in \cite[Thm.~E]{chmy} that any compact constant mean curvature $C^2$ surface in $(\stres,g_h)$ is homeomorphic to the sphere $\sph^2$ or the torus $\sph^1\times\sph^1$. With the aim of discarding toroidal surfaces as boundaries of isoperimetric regions we considered stable surfaces, i.e., those minimizing the area up to second order for volume-preserving variations. In this direction, we showed in \cite[Sect.~5]{hr2} that the spheres $\s_\la$ are stable surfaces in $(\stres,g_h)$. As a matter of fact, in the recent work \cite[Sect.~5]{hr4} we have characterized the spheres $\s_\la$ as the unique complete, embedded and stable $C^2$ surfaces in $(\stres,g_h)$. This implies in particular that the aforementioned Pansu conjecture in $(\stres,g_h)$ is true under $C^2$ regularity of the isoperimetric solutions, see \cite[Sect.~6]{hr4}. Unfortunately, though the spheres $\s_\la$ with $\la>0$ are $C^2$ smooth (but not $C^3$ around the poles), the $C^2$ regularity of isoperimetric boundaries in sub-Riemannian $3$-manifolds is a difficult open question, even in the first Heisenberg group.  

Our objective in these notes is to provide a new evidence supporting Pansu's conjecture in $(\stres,g_h)$, as that as an area-minimizing property for the closed half-spheres of $\s_0$ spanned by a vertical great circle. Let us give a detailed description of our motivations, results and techniques.

The surface $\s_0$ is a totally geodesic $2$-sphere of $\stres$ which is critical for the sub-Riemannian area but does not minimize since its deformation by Riemannian equidistants decreases the area~\cite[Re.~5.8]{hr2}. However, the sphere $\s_0$ is a second order minimum of the area under variations fixing the equator $C_0:=\{0\}\times\sph^1$, see \cite[Prop.~5.5]{hr2}. Thus, it is natural to ask if the corresponding half-spheres $\s^+_0$ and $\s^-_0$ with boundary $C_0$ solve the Plateau problem in $(\stres,g_h)$ of minimizing the area among compact surfaces with the same boundary. In Theorem~\ref{th:main} we answer positively this question for $C^1$ competitors and prove the uniqueness statement that, up to vertical translations, $\s_0^+$ is the only minimizer of class $C^2$ for this problem. We remark that existence and uniqueness of solutions for the sub-Riemannian mean curvature operator with Dirichlet boundary condition have been analyzed in previous works, see for instance \cite{chmy,chy,cchy2}.

As to our isoperimetric result, the motivation comes from a previous theorem in the Heisenberg group, where Ritor\'e~\cite[Sect.~3]{ritore-calibrations} deduced that the Pansu spheres minimize the perimeter for fixed volume among those sets inside a vertical right cylinder that contain a horizontal disk. In a recent work, Pozuelo and Ritor\'e~\cite[Sect.~6]{pozuelo-ritore} have generalized this property for Pansu-Wulff shapes in the Heisenberg group. In Theorem~\ref{th:main2} we establish that the topological $3$-balls $\be_\la$ enclosed by the spherical surfaces $\s_\la$ with $\la>0$ uniquely solve the isoperimetric problem in $(\stres,g_h)$ in the family of $C^1$ sets within a vertical solid tube and containing a horizontal section of the tube. Here, a \emph{vertical solid tube} $\mathcal{W}_\mu$ is a metric tube around the great circle $L:=\sph^1\times\{0\}$, see equation \eqref{eq:wla}; in particular, the boundary $\ptl\mathcal{W}_\mu$ is the Clifford torus of points in $\stres$ at Riemannian distance $\arccos(1/\mu)$ from $L$.

The proofs of our area comparisons in Theorems~\ref{th:main} and \ref{th:main2} rely on calibration arguments. These have been extensively employed in the Calculus of Variations and provide a powerful tool to derive area-minimizing properties of minimal and constant mean curvature surfaces, see for example \cite[Sect.~1.3.1]{ritore-libro} and the references therein. In sub-Riemannian geometry, the use of calibrations have appeared in different contexts, see for instance the works of Cheng, Hwang, Malchiodi and Yang~\cite[Sect.~6]{chmy}, Barone Adesi, Serra Cassano and Vittone \cite[Sect.~2]{bscv}, Ritor\'e and the second author~ \cite[Sect.~5]{rr2}, Ritor\'e \cite[Sect.~3]{r2}, \cite[Sect.~3]{ritore-calibrations} and Pozuelo and Ritor\'e~\cite[Sect.~6]{pozuelo-ritore}. 

To construct calibrations in our setting we proceed as follows, see Proposition~\ref{prop:foliation} for the details. For a given half-sphere $\s_\la^+$ we use translations along the fibers of the Hopf fibration to produce a family $\{\s_\la^+(t)\}_{t\in[0,2\pi)}$ of surfaces with constant mean curvature $\la$ that foliates the vertical solid tube $\mathcal{W_\la}$ (this one coincides with $\sph^3-C_0$ for $\la=0$). Then, the calibration vector field $X^+_\la$ is defined on $\mathcal{W}_\la-L$ as the horizontal Gauss map along the leaves of the foliation, so that its Riemannian divergence equals $-2\la$. By repeating this construction with the half-spheres $\s^-_\la$ we obtain another calibration vector field $X^-_\la$. It is worth mentioning that $X^+_\la$ and $X^-_\la$ cannot be extended continuously to $L$ when $\la>0$ nor to $C_0\cup L$ when $\la=0$. This difficulty is overcome by means of an approximation argument so that, after some estimates, we deduce our area comparisons by applying the divergence theorem to these vector fields over suitable regions and passing to the limit.

On the other hand, the characterization of equality cases is considerably easier in Theorem~\ref{th:main2}. The main reason is that the boundary $\Sg$ of any set $\Om$ in the conditions of the statement is tangent to a spherical surface $\s_\la$ along a circle $C_\la$ of non-singular points. By using this, we can conclude from the equality case that $\Sg$ and $\s_\la$ coincide in a small neighborhood of $C_\la$ in $\Sg$. From here, a completeness argument shows that $\Sg=\s_\la$, as desired. The discussion of equality in Theorem~\ref{th:main} is more difficult because the hypotheses are less rigid. Indeed, if $\Sg$ is a $C^1$ solution to the Plateau problem in $(\stres,g_h)$ with boundary $C_0$ then, locally around $C_0$, the surface $\Sg$ is union of small geodesic segments that leave from $C_0$ in the horizontal direction determined by an angle function $\sg$. Since the half-sphere $\s^+_0(t)$ corresponds to a special choice of $\sg$, a careful analysis of the surface $\Sg$ as a function of $\sg$ is required to deduce that $\Sg=\s^+_0(t)$ for some $t\in[0,2\pi)$, see Proposition~\ref{prop:uniqueplateau}. In our analysis, that partly depends on the previous work~\cite[Sect.~4]{hr3} where we studied how $\sg$ influences the geometry of $\Sg$, the main difficulty is the possible existence of curves of singular points. This issue has appeared frequently in previous uniqueness results for the sub-Riemannian minimal surface equation with Dirichlet boundary condition. In our proof we are able to discard singular curves by assuming the $C^2$ regularity of $\Sg$, which allows us to employ the $C^1$ regularity of these curves and the local behaviour of $\Sg$ near them, see \cite{chmy,rr2,hr1}. Unfortunately, the singular set for $C^1$ solutions of the Plateau problem can be much more complicated, as was shown in the first Heisenberg group by Cheng, Hwang, Malchiodi and Yang~\cite{chmy2}. By this reason, our argument in Proposition~\ref{prop:uniqueplateau} does not hold directly for $C^1$ surfaces, though the characterization of equality in this situation would be probably the same. 

Finally, we must point out that the optimal regularity for the Plateau problem has not yet been established, even in the Heisenberg group, where it is possible to find examples with low regularity~\cite{pauls-regularity,chy,mscv,r2}. However, for viscosity solutions in the Heisenberg group the work of Capogna, Citti and Manfredini~\cite{MR2583494} implies $C^{1,\alpha}$ regularity with $\alpha\in (0,1)$. 

The paper is organized into three sections. In Section~\ref{sec:preliminaries} we gather some preliminary material about the sub-Riemannian sphere $(\stres,g_h)$ and the geometry of surfaces in $\stres$. In Section~\ref{sec:foliation} we recall some facts about the spherical surfaces $\{\s_\la\}_{\la\geq 0}$ that we use to construct foliations of vertical solid tubes and the associated calibration vector fields. Section~\ref{sec:plateau} is devoted to the Plateau problem for the circle $C_0$. Finally, in Section~\ref{sec:isoperimetric} we prove the isoperimetric property of the spheres $\s_\la$.

\section{Preliminaries}
\label{sec:preliminaries}
\setcounter{equation}{0}

Here we introduce the geometric setup and gather some basic facts that will be used throughout the paper. We will mostly follow the notation employed in \cite{hr1}.

\subsection{The sub-Riemannian $3$-sphere}
\label{subsec:sphere}

In the Euclidean space $\rr^4$ we identify a point $p=(x_{1},y_{1},x_{2},y_{2})$ with the quaternion $x_{1}+iy_{1}+jx_{2}+ky_{2}$. Given $p,q\in\rr^4$ we denote by $p\cdot q$ and $\escpr{p,q}$ the quaternion product and the scalar product of $p$ and $q$, respectively. The sphere $\stres\sub\rr^4$ is the set of unit quaternions. The pair $(\stres,\cdot)$ is a compact Lie group with identity element $e:=(1,0,0,0)$. The restriction of the scalar product $\escpr{\cdot\,,\cdot}$ to the tangent bundle $T\stres$ provides the Riemannian metric $g$ on $\stres$ of constant sectional curvature $1$. The associated Riemannian distance $d$ in $(\stres,g)$ can be computed as $d(p,q)=\arccos(\escpr{p,q})$.

For any $p\in\stres$, the \emph{right and left translations} by $p$ are the diffeomorphisms defined by $q\mapsto q\cdot p$ and $q\mapsto p\cdot q$, respectively. These translations are isometries of $(\stres,g)$ since the Riemannian metric $g$ is bi-invariant.  A basis of right invariant vector fields in $(\stres,\cdot)$ is 
\begin{align}
\label{eq:hopf}
T(p):&=i\cdot p=-y_{1}\,\frac{\ptl}{\ptl x_{1}}+x_{1}\,\frac{\ptl}{\ptl y_{1}}
-y_{2}\,\frac{\ptl}{\ptl x_{2}}+x_{2}\,\frac{\ptl}{\ptl y_{2}},
\\
\nonumber
E_1(p):&=j\cdot p=-x_{2}\,\frac{\ptl}{\ptl x_{1}}+y_{2}\,\frac{\ptl}{\ptl
y_{1}}+x_{1}\,\frac{\ptl}{\ptl x_{2}}-y_{1}\,\frac{\ptl}{\ptl y_{2}},
\\
\nonumber
E_{2}(p):&=k\cdot p=-y_{2}\,\frac{\ptl}{\ptl x_{1}}-x_{2}\,\frac{\ptl}{\ptl
y_{1}}+y_{1}\,\frac{\ptl}{\ptl x_{2}}+x_{1}\,\frac{\ptl}{\ptl y_{2}},
\end{align}
where $\ptl/\ptl x_i$ and $\ptl/\ptl y_i$ with $i=1,2$ are the Euclidean coordinate vector fields in $\rr^4$. Note that $\{E_1,E_2,T\}$ is an orthonormal basis of $T\stres$ with respect to the Riemannian metric $g$. 

We call \emph{horizontal distribution} in $\stres$ to the smooth planar distribution $\h$ generated by $E_{1}$ and $E_{2}$. Note that $\h=\text{Ker}(\sigma)$, where $\sigma$ is the contact $1$-form in $\stres$ given by $\sigma:=-y_{1}\,dx_{1}+x_{1}\,dy_{1}-y_{2}\,dx_{2}+x_{2}\,dy_{2}$. This implies that $\h$ is completely nonintegrable. We consider the orientation in $\h$ (resp. in $T\stres$) for which $\{E_1,E_2\}$ (resp. $\{E_1,E_2,T\}$) is a positive basis.  The oriented volume form $dv$ in $(\stres,g)$ equals $\frac{1}{2}\,\sigma\wedge (d\sigma)$. More precisely $(dv)(u_1,u_2,u_3)=\text{det}(u_1,u_2,u_3)$, where the coordinates of $u_i$ are taken with respect to $\{E_1,E_2,T\}$. The \emph{horizontal projection} of a tangent vector $X$ is denoted by $X_{h}$.  A vector $X$ on $\stres$ is \emph{horizontal} if $X=X_h$. In case $X$ is proportional to $T$ then we say that $X$ is \emph{vertical}.

The \emph{sub-Riemannian metric} $g_h$ on $\stres$ is the restriction to $\h$ of the Riemannian metric $g$. In particular, $\{E_{1},E_{2}\}$ provides a positive orthonormal basis of $\h$ with respect to $g_h$. By an \emph{isometry} of $(\stres,g_h)$ we mean a diffeomorphism $\phi:\stres\to\stres$ whose differential at any $p\in\stres$ is an orientation-preserving linear isometry between $\h_p$ and $\h_{\phi(p)}$. 
We say that two subsets $\Om_{1}$ and $\Om_{2}$ in $\stres$ are \emph{congruent} if there is an isometry $\phi$ of $(\stres,g_h)$ such that $\phi(\Om_{1})=\Om_{2}$.

For any tangent vector $X$ on $\stres$ we denote $J(X):=D_XT$, where $D$ is the Levi-Civita connection in $(\stres,g)$. The restriction of $J$ to $\h$ coincides with the orientation-preserving 90 degree rotation on $\h$, see \cite[Sect.~2]{hr1} for details. In particular, we get
\begin{equation}
\label{eq:jota}
g_h(J(X),X)=0, \quad \text{for any } X\in\h.
\end{equation}
We remark that $J(X)=i\cdot X$, for any horizontal vector $X$.

Sometimes we will identify $\stres$ with the set $\{(z_1,z_2)\in\mathbb{C}^2\,;\,|z_1|^2+|z_2|^2=1\}$, where $|z|$ is the modulus of the complex number $z$. If we consider the circle
\begin{equation}
\label{eq:ele}
L:=\sph^1\times\{0\},
\end{equation}
then the Riemannian distance $d_L(p)$ in $(\stres,g)$ between $p\in\stres$ and $L$ can be computed as
\begin{equation}
\label{eq:distance}
d_L(p)=\arccos(|z_1|), \quad\text{for any } p=(z_1,z_2)\in\stres.
\end{equation}
In particular $d_L$ is a $C^\infty$ function whenever $0<|z_1|<1$. Note also that the circle $C_0:=\{0\}\times\sph^1$ is characterized by equality
\[
C_0=\{p\in\stres\,;\,d_L(p)=\pi/2\}.
\]
On the other hand, the \emph{rotation} $r_\theta:\stres\to\stres$ defined by
\begin{equation}
\label{eq:rtheta}
r_{\theta}(z_1,z_2):=(z_1,\exp(i\theta)\cdot z_2)
\end{equation}
is an isometry of $(\stres,g_h)$ fixing $L$, since it carries the orthonormal basis
$\{E_{1},E_{2},T\}$ at $p$ to the orthonormal basis
$\{\cos(\theta)\,E_{1}
+\sin(\theta)\,E_{2},-\sin(\theta)\,E_{1}+\cos(\theta)\,E_{2},T\}$ at
$r_{\theta}(p)$.  We say that a set $\Om\subseteq\stres$ is \emph{rotationally invariant} if $r_\theta(\Om)=\Om$ for any $\theta\in [0,2\pi)$.

\subsection{Vertical translations and Clifford tori}
\label{subsec:hopf}

The \emph{Hopf fibration} is the Riemannian submersion $\mathcal{F}:\stres\to\stres\cap\{x_1=0\}$  given by
$\mathcal{F}(p):=\overline{p}\cdot i\cdot p$, where $\overline{p}$ is the conjugate of the quaternion $p$.  If we write $p=(z_1,z_2)$ with $|z_1|^2+|z_2|^2=1$, then we have
\[
\mathcal{F}(p)=\big(0,|z_1|^2-|z_2|^2,2\,\text{Im}(z_1\cdot\overline{z}_2),2\,\text{Re}(z_1\cdot\overline{z}_2)\big).
\]
The fiber of $\mathcal{F}$ through $p\in\stres$ is the great circle parameterized by $\varphi_t(p)$ with $t\in\rr$, where 
\begin{equation}
\label{eq:verttrans}
\varphi_t(p):=\exp(it)\cdot p.
\end{equation} 
The family of maps $\{\varphi_t\}_{t\in\rr}$ provides the one-parameter group of diffeomorphisms associated to the vector field $T$ in \eqref{eq:hopf}, which is usually called the \emph{vertical Hopf vector field} in $\stres$. We will refer to the left translations $\{\varphi_t\}_{t\in\rr}$ as \emph{vertical translations}. The \emph{vertical axis} is the great circle $L$ in \eqref{eq:ele}, which is parameterized by the integral curve of $T$ through the identity element $e$. Note that any vertical translation $\varphi_t$ is an isometry of $(\stres,g_h)$ with $\var_t(L)=L$ and $\var_t(C_0)=C_0$, where $C_0:=\{0\}\times\sph^1$.

For any $\rho\in (0,1)$ we consider the geodesic circle $c_{\rho}:=\sph^3\cap\{x_{1}=0,\,y_1=2\rho^2-1\}$. The set $\mathcal{F}^{-1}(c_{\rho})$ is the \emph{vertical Clifford torus} $\ttt_\rho:=\sph^1(\rho)\times\sph^1(\sqrt{1-\rho^2})$ in $\stres$. Note that the vertical vector field $T$ is tangent over $\ttt_\rho$. By taking into account \eqref{eq:distance} it follows that $\ttt_\rho$ is a tube around $L$, in the sense that
\begin{equation}
\label{eq:clifford}
\ttt_\rho=\{p\in\stres\,;\,d_L(p)=\arccos(\rho)\}.
\end{equation}
It is also clear by \eqref{eq:rtheta} that $\ttt_\rho$ is a rotationally invariant surface.

\subsection{Cylindrical coordinates}
\label{subsec:cylcoor}

In the closed $2$-dimensional half-sphere
\[
\sph^2_+:=\stres\cap\{x_2\geq 0, y_2=0\}
\] 
we take \emph{spherical coordinates} $(\omega,\tau)$, with $\omega\in [0,\pi/2]$ and $\tau\in[0,2\pi)$. In complex notation any point in $\sph^2_+$ can be written as $(\cos(\omega)\cdot\exp(i\,\tau),\sin(\omega))$. For any  $p=(z_1,z_2)\in\stres$ it is clear that $p=r_\vartheta(p')$ for some $\vartheta\in[0,2\pi)$, where $p':=(z_1,|z_2|)\in\sph^2_+$ and $r_\vartheta$ is the rotation defined in \eqref{eq:rtheta}. 
So, we can find $\omega\in[0,\pi/2]$ and $\tau\in[0,2\pi)$ such that $p$ is expressed in \emph{cylindrical coordinates} $(\omega,\tau,\vartheta)$ as
\begin{equation}
\label{eq:cylindrical}
p=(\cos(\omega)\cdot\exp(i\,\tau),\sin(\omega)\cdot\exp(i\,\vartheta)).
\end{equation}
It follows from \eqref{eq:distance} that $\omega=d_L(p)$, where $L$ is the vertical axis. So, the coordinate $\omega$ is uniquely determined by $p$. On the other hand, equality \eqref{eq:cylindrical} also holds for $\tau'=\tau+2n\pi$ and $\vartheta'=\vartheta+2m\pi$ with $n,m\in\mathbb{Z}$. From \eqref{eq:rtheta} and \eqref{eq:verttrans}, if $p$ has cylindrical coordinates $(\omega,\tau,\vartheta)$, then $r_\theta(p)$ and $\var_t(p)$ have cylindrical coordinates $(\omega,\tau,\vartheta+\theta)$ and $(\omega,\tau+t,\vartheta+t)$, respectively (the values of $\tau$ and $\vartheta$ are determined modulo $2\pi$). 

If $\Sg$ is a rotationally invariant surface in $\stres$, then there is a generating curve of $\Sg$ in $\sph^2_+$ with associated spherical coordinates $(\omega,\tau)$. For the vertical Clifford torus $\mathcal{T}_\rho$ in \eqref{eq:clifford}, the generating curve is the circle of equation $\omega=\arccos(\rho)$. For the totally geodesic $2$-sphere in $(\stres,g)$ obtained by intersection with the hyperplane of equation $\sin(\tau_0)\,x_1-\cos(\tau_0)\,y_1=0$ with $\tau_0\in[0,\pi)$, the generating curve is described by equations $\tau=\tau_0$ and $\tau=\tau_0+\pi$.

\subsection{Horizontal geometry of surfaces}
\label{subsec:geometry}

Let $\Sg$ be a $C^1$ surface in $\stres$ (we will suppose that $\Sg$ is embedded unless otherwise stated). The \emph{singular set} $\Sg_0$ consists of the points $p\in\Sg$ for which the tangent plane $T_p\Sg$ equals the horizontal plane $\h_{p}$. Since $\h$ is completely nonintegrable, it follows by Frobenius theorem that $\Sg_0$ is closed and has empty interior in $\Sg$. Hence the \emph{regular set} $\Sg-\Sg_0$ is open and dense in $\Sg$. If $N$ is a unit normal vector field to $\Sg$ in $(\stres,g)$ then $\Sg_0=\{p\in\Sg\,;\,N_h(p)=0\}$. In the regular set $\Sg-\Sg_0$ we define the \emph{horizontal Gauss map} $\nu_h$ and the \emph{characteristic vector field} $Z$ by
\begin{equation*}
\label{eq:nuh}
\nu_h:=\frac{N_h}{|N_h|}, \qquad Z:=J(\nuh).
\end{equation*}
Note that $Z$ is horizontal and orthogonal to $\nu_h$ by \eqref{eq:jota}. In particular, $Z$ is always tangent to $\Sg$. Moreover, for any $p\in\Sg-\Sg_0$, we have the
orthonormal basis of $T_{p}\stres$ given by $\{Z(p),\nuh(p),T(p)\}$. As a consequence, we deduce the identity
\begin{equation}
\label{eq:normal}
N=\mnh\,\nuh+\escpr{N,T}\,T\quad\text{on } \Sg-\Sg_0.
\end{equation}
The integral curves of $Z$ in $\Sg-\Sg_0$ are called (oriented) \emph{characteristic curves} of $\Sg$. If $\Sg$ is a $C^2$ surface then $Z$ is a $C^1$ vector field and the maximal characteristic curves provide a foliation of $\Sg-\Sg_0$.

\subsection{Mean curvature and area-stationary surfaces}
\label{subsec:stationary}

The \emph{volume} of a set $\Om\subeq\stres$ is the Riemannian volume $V(\Om)$ in $(\stres,g)$. The (sub-Riemannian) \emph{area} of a $C^1$ surface $\Sg$ in $\stres$ is defined by
\begin{equation}
\label{eq:area}
A(\Sg):=\int_\Sg\mnh\,da,
\end{equation}
where $N$ is a unit normal vector field to $\Sg$ in $(\stres,g)$ and $da$ is the area element in $(\stres,g)$.  

We say that a surface $\Sg$ is \emph{area-stationary} if it is a critical point of the area functional \eqref{eq:area} for any compactly supported variation $\Sg_\alpha$ of $\Sg$. If $\Sg$ encloses a set $\Om$ and has critical area under variations $\Sg_\alpha$ for which the enclosed sets $\Om_\alpha$ have the same volume, then we say that $\Sg$ is \emph{volume-preserving area-stationary}. If $\Sg$ is an orientable area-stationary (resp. volume-preserving area-stationary) $C^2$ surface, then the mean curvature $H$ vanishes (resp. is constant), see \cite[Sect.~2]{chmy}, \cite[Sect.~4]{hr1} and \cite[Sect.~4.1]{hr2}.  Here the (sub-Riemannian) \emph{mean curvature} of $\Sg$ is the quantity
\begin{equation}
\label{eq:mc}
-2H(p):=(\divv_\Sg\nuh)(p), \quad\text{for any } p\in\Sg-\Sg_0,
\end{equation}
where $\divv_\Sg$ stands for the divergence relative to $\Sg$ in $(\stres,g)$. Moreover, we can reason as in \cite[Cor.~4.11]{rr2}, see also \cite[Sect.~4.1]{hr2}, to get the identity
\begin{equation}
\label{eq:eso}
\frac{d}{d\alpha}\bigg|_{\alpha=0}\big(A(\Sg_\alpha)-2H\,V(\Om_\alpha\big))=0,
\end{equation}
for any variation of $\Sg$ by surfaces $\Sg_\alpha$ enclosing sets $\Om_\alpha$.

\section{Pansu spheres and foliations of vertical solid tubes}
\label{sec:foliation}

We begin this section by recalling some facts about the family of constant mean curvature spheres in $(\stres,g_h)$ introduced in \cite[Sect.~5.1]{hr1}. After that, we will apply vertical translations to a fixed closed half-sphere to produce a foliation of the associated solid tube around the vertical axis. These foliations will be employed to construct the calibration vector fields in the proofs of our main results in Sections~\ref{sec:plateau} and \ref{sec:isoperimetric}.

For any number $\la\geq 0$ we define parametrically the set
\begin{equation}
\label{eq:sla}
\s_\la:=\{\ga_\theta(s)\,;\,\theta\in[0,2\pi), s\in[0,\pi/\sqrt{1+\la^2}]\},
\end{equation}
where $\ga_\theta(s):=(x_1(s),y_1(s),x_2(\theta,s),y_2(\theta,s))$ is the curve in $\stres$ with coordinate functions
\begin{equation}
\label{eq:geocoor}
\begin{split}
x_{1}(s):&=\cos(\la s)\,\cos(\sqrt{1+\la^2}\,s)+
\frac{\la}{\sqrt{1+\la^2}}\,\sin(\la s)\,\sin(\sqrt{1+\la^2}\,s),
\\
y_{1}(s):&=-\sin(\la s)\,\cos(\sqrt{1+\la^2}\,s)+
\frac{\la}{\sqrt{1+\la^2}}\,\cos(\la s)\,\sin(\sqrt{1+\la^2}\,s),
\\
x_{2}(\theta,s):&=\frac{1}{\sqrt{1+\la^2}}\,\sin(\sqrt{1+\la^2}\,s)\,
\cos(\theta-\la s),
\\
y_{2}(\theta,s):&=\frac{1}{\sqrt{1+\la^2}}\,\sin(\sqrt{1+\la^2}\,s)\,
\sin(\theta-\la s).
\end{split}
\end{equation}
It comes from \cite[Sect.~3]{hr1} that $\ga_\theta$ is the \emph{geodesic} in $(\stres,g_h)$ of curvature $\la$ leaving from the \emph{south pole} $e:=(1,0,0,0)$ with initial velocity $\cos(\theta)\,E_1(e)+\sin(\theta)\,E_2(e)$. Note that $\ga_\theta(\pi/\sqrt{1+\la^2})\in L$ and does not depend on $\theta$. We denote this point by $e^\la$ and we call it the \emph{north pole} of $\s_\la$. For $\la=0$ any curve $\ga_\theta(s)$ with $s\in[0,\pi]$ parameterizes an arc between $e$ and $-e$ of a great circle in $\stres$. Indeed, we have $\s_0=\stres\cap\{y_1=0\}$, which is a totally geodesic 2-sphere in $(\stres,g)$. 

In the next result we gather some properties of $\s_\la$ that will be useful in the sequel.

\begin{proposition}
\label{prop:spheres}
For any $\la\geq0$ the following facts hold:
\begin{itemize}
\item[(i)] $\s_\la$ is an embedded $2$-sphere of class $C^2$ with singular set $\{e,e^\la\}$,
\item[(ii)] $\s_\la-\{e,e^\la\}$ is a $C^\infty$ surface,
\item[(iii)] there is a $C^1$ unit normal vector field $N$ on $\s_\la$ in $(\stres,g)$, which is $C^\infty$ outside the poles, and such that $\s_\la$ has constant mean curvature $\la$ in $(\stres,g_h)$ with respect to $N$,
\item[(iv)] the curves $\ga_\theta(s)$ with $s\in (0,\pi/\sqrt{1+\la^2})$ are characteristic curves of $\s_\la$,
\item[(v)] the surface $\s_\la$ is volume-preserving area-stationary. The sphere $\s_0$ is also area-stationary,
\item[(vi)] $\s_\la$ is a rotationally invariant surface. For any $\la>0$ the generating curve of $\s_\la$ in spherical coordinates $(\omega,\tau)$ is the union of two graphs $\tau=\mu_\la(\omega)$ and $\tau=\tau_\la-\mu_\la(\omega)$, where $\tau_\la:=\pi\,\big(1-\frac{\la}{\sqrt{1+\la^2}}\big)$ and $\mu_\la:[0,\arctan(1/\la)]\to\rr$ is a continuous increasing function with $\mu_\la(0)=0$ and $\mu_\la(\arctan(1/\la))=\tau_\la/2$.
\item[(vii)] the volume function $m:(0,+\infty)\to (0,V(\stres)/2)$ defined by $m(\la):=V(\be_\la)$ is decreasing, where $\be_\la$ is the topological closed $3$-ball obtained by vertical rotations of the closed set in $\sph^2_+$ enclosed between the generating curve of $\s_\la$ and the arc $0\leq\tau\leq\tau_\la$.
\end{itemize}
\end{proposition}

\begin{proof}
Statements (i)-(v) were proved in \cite[Prop.~5.1, Re.~5.2]{hr1}. Statement (vi) comes from the description of $\s_\la$ in the proof of \cite[Thm.~6.4 (iii)]{hr1}. Indeed, the function $\mu_\la(\omega)$ is explicitly given by $\tau_\la/2-\rho_\la(\omega)$, where
\[
\rho_\la(\omega):=\arccos(\la\,\tan\,\omega)-\frac{\la}{\sqrt{1+\la^2}}\,\arccos(\sqrt{1+\la^2}\,\sin\,\omega), \quad\omega\in[0,\arctan(1/\la)].
\]
Hence the generating curve of $\s_\la$ in $\sph^2_+$ coincides, up to a translation along the $\tau$-axis, with the set $\mathcal{G}_\la$ of equation $\tau=\pm\rho_\la(\omega)$ with $\omega\in[0,\arctan(1/\la)]$. 
For fixed $\la_0>0$ and $\omega\in[0,\arctan(1/\la_0))$ we have that $\omega\in[0,\arctan(1/\la))$ for any $\la\in(0,\la_0)$, and the function $\la\in(0,\la_0)\mapsto\rho_\la(\omega)$ is decreasing. This entails that the closed set in $\sph^2_+$ enclosed between $\mathcal{G}_\la$ and the arc $-\tau_\la/2\leq\tau\leq\tau_\la/2$ is decreasing with respect to $\la$. Therefore, the function $\la\mapsto V(\mathcal{B}_\la)$ is decreasing in $(0,+\infty)$. This shows statement (vii) and finishes the proof.
\end{proof}

For our next result we need to introduce some notation. We define $\be_0$ as the closure in $\stres$ of the connected component of $\stres-\s_0$ that contains all the sets $\s_\la-\{e\}$ with $\la>0$. For any $\la\geq 0$ the \emph{equatorial circle} of $\s_\la$ is the curve
\[
C_\la:=\{\ga_\theta\big(\pi/(2\sqrt{1+\la^2})\big)\,;\,\theta\in [0,2\pi)\},
\] 
which is obtained by applying the rotations $r_\theta$ in \eqref{eq:rtheta} to the point in the generating curve of $\s_\la$ with spherical coordinates $(\arctan(1/\la),\tau_\la/2)$ (when $\la=0$ we can take $(\pi/2,\pi/2)$). Let $\sph^2_\la$ be the rotationally invariant and totally geodesic $2$-sphere in $(\stres,g)$ with $C_\la\sub\sph^2_\la$ (for $\la=0$ we have $\sph^2_\la=\stres\cap\{x_1=0\}$). We say that  $D_\la:=\be_\la\cap\sph^2_\la$ is the \emph{equatorial disk} of $\be_\la$. Clearly $\ptl D_\la=C_\la$. The $2$-sphere $\sph^2_\la$ separates $\stres$ into two closed regions $(\stres_\la)^-$ and $(\stres_\la)^+$, that we label in such a way that $e\in(\sph_\la^3)^+$. For any set $\Om\subeq\stres$ we denote
\begin{equation}
\label{eq:halves}
\Om_\la^-:=\Om\cap (\stres_\la)^- \ \text{ and } \ \Om_\la^+:=\Om\cap (\stres_\la)^+.
\end{equation}
For any $\la>0$ the \emph{vertical solid tube} around $L$ associated to $\s_\la$ is the set
\begin{equation}
\label{eq:wla}
\mathcal{W}_\la:=\{p\in\stres\,;\,d_L(p)\leq\arctan(1/\la)\}.
\end{equation}
From \eqref{eq:clifford} it is clear that $\ptl\mathcal{W}_\la=\mathcal{T}_\rho$, where $\rho\in (0,1)$ satisfies $\arccos(\rho):=\arctan(1/\la)$. Note that $\s_\la$ and $\mathcal{T}_\rho$ are tangent surfaces along the circle $C_\la$. We also denote
\begin{equation*}
\label{eq:w0}
\mathcal{W}_0=\stres-C_0,
\end{equation*}
where $C_0=\{0\}\times\sph^1$ is the equatorial circle of $\s_0$.

Now, we are ready to prove the following statement.

\begin{proposition}
\label{prop:foliation}
Consider a closed half-sphere $\s_\la^+$ for some $\la\geq 0$. For any $t\in [0,2\pi)$, let $\s_\la^+(t):=\varphi_t(\s_\la^+)$, where $\varphi_t$ is the vertical translation defined in \eqref{eq:verttrans}. Then, we have:
\begin{itemize}
\item[(i)] $\s_\la^+(t)$ is a rotationally invariant surface with boundary, having constant mean curvature $\la$ in $(\stres,g_h)$ and one singular point  $e_t:=\varphi_t(e)$,
\item[(ii)] for $\la>0$ the family $\{\s_\la^+(t)\}_{t\in [0,2\pi)}$ produces a foliation of $\mathcal{W}_\la$,
\item[(iii)] the family $\{\s_0^+(t)-C_0\}_{t\in [0,2\pi)}$ provides a foliation of $\mathcal{W}_0$,
\item[(iv)] there is a unit horizontal $C^\infty$ vector field $X^+_\la$ on $\mathcal{W}_\la-L$ whose restriction to $\s_\la^+$ equals the horizontal Gauss map, and having Riemannian divergence
\[
\divv X^+_\la=-2\la \ \text{ on } \ \mathcal{W}_\la-L.
\]
\end{itemize}
The above properties remain valid if we replace $\s_\la^+$ with $\s_\la^-$; the only difference is that the unique singular point of $\s_\la^-(t):=\varphi_t(\s_\la^-)$ is $\varphi_t(e^\la)$.
\end{proposition}

\begin{proof}
By Proposition~\ref{prop:spheres} (iii) there is a $C^1$ unit normal $N_0$ on $\s_\la^+(0)=\s_\la^+$ in $(\stres,g)$, which is $C^\infty$ on $\s_\la^+-\{e\}$, and having associated mean curvature $\la$ in $(\stres,g_h)$. For any $t\in[0,2\pi)$ we define $N_t:=\var_t\circ N_0\circ\var_t^{-1}$. This is a Riemannian unit normal on $\s_\la^+(t)$ since $\var_t$ is an isometry of $(\stres,g_h)$. Moreover, $\s_\la^+(t)$ has constant mean curvature $\la$ in $(\stres,g_h)$ with respect to $N_t$ and only one singular point at $e_t=\var_t(e)$. As $\s_\la^+$ is rotationally invariant then
\[
r_\theta(\s_\la^+(t))=(r_\theta\circ\var_t)(\s_\la^+)=(\var_t\circ r_\theta)(\s_\la^+)=\var_t(\s_\la^+)=\s_\la^+(t),
\] 
for any rotation $r_\theta$ in \eqref{eq:rtheta}, so that $\s_\la^+(t)$ is also rotationally invariant. This proves (i).

For $\la>0$ we know by Proposition~\ref{prop:spheres} (vi) that the generating curve of $\s_\la^+(0)$ in spherical coordinates is the graph $(\omega,\mu_\la(\omega))$ where $\omega\in [0,\arctan(1/\la)]$. Thus, the graph $(\omega,\mu_\la(\omega)+t)$ determines the generating curve of $\s_\la^+(t)$. This family of curves gives a foliation of the set $\{p\in\sph^2_+\,;\,d_L(p)\leq\arctan(1/\la)\}$. On the other hand, the generating curves of the half-spheres $\s_0^+(t)$ are circle arcs of spherical coordinates $\tau=t$, so that they produce a foliation of $\sph^2_+-C_0$ (note that $C_0=\ptl\s_0^+(t)$ for any $t\in [0,2\pi)$). Now the proofs of (ii) and (iii) come by applying vertical rotations $r_\theta$ and having in mind that $\mathcal{W}_\la$ is rotationally invariant.

Finally, we will construct the vector field in (iv). Observe that the set $\{e_t\,;\,t\in[0,2\pi)\}$ of the singular points of $\s_\la^+(t)$ parameterizes the vertical axis $L$. From statements (ii) and (iii) we can define a vector field $N_\la$ on $\mathcal{W}_\la$ whose restriction to $\s_\la^+(t)$ equals the unit normal $N_t$. This vector field is $C^1$ on $\mathcal{W}_\la$ and $C^\infty$ on $\mathcal{W}_\la-L$. Let $X^+_\la:=\frac{(N_t)_h}{|(N_t)_h|}$, which is a horizontal unit $C^\infty$ vector field on $\mathcal{W}_\la-L$ such that $X^+_\la=(\nu_t)_h$ along $\s_\la^+(t)-\{e_t\}$. Let us compute $\divv X^+_\la$. For $p\in\mathcal{W}_\la-L$ we consider the unique $t\in[0,2\pi)$ such that $p\in\s_\la^+(t)$. Note that
\[
(\divv X^+_\la)(p)=\big(\!\divv_{\s_\la^+(t)}(\nu_t)_h\big)(p)+\escpr{D_{N_t(p)}X^+_\la,N_t(p)}=-2\la+\escpr{D_{N_t(p)}X^+_\la,N_t(p)},
\]
where we have used \eqref{eq:mc} and that $\s_\la^+(t)$ has mean curvature $\la$ with respect to $N_t$. Next, we see that the last summand above vanishes. From \eqref{eq:normal} and the definition of $X^+_\la$, we infer
\begin{align*}
\escpr{D_{N_t(p)}X^+_\la,N_t(p)}&=|(N_t)_h|(p)\,\escpr{D_{N_t(p)}X^+_\la,X^+_\la(p)}+\escpr{N_t(p),T(p)}\,\escpr{D_{N_t(p)}X^+_\la,T(p)}
\\
&=-\escpr{N_t(p),T(p)}\,\escpr{X^+_\la(p),D_{N_t(p)}T}
\\
&=-|(N_t)_h|(p)\,\escpr{N_t(p),T(p)}\,\escpr{(\nu_t)_h(p),Z_t(p)}=0,
\end{align*}
as we claimed. All the previous arguments also hold when we replace $\s_\la^+$ with $\s_\la^-$.
\end{proof}

\section{The Plateau problem for the circle $C_0$}
\label{sec:plateau}

In this section we use a calibration argument to show that the closed half-sphere $\s_0^+$ minimizes the area in $(\stres,g_h)$ among compact surfaces with boundary $C_0=\{0\}\times\sph^1$. Related minimization results in different sub-Riemannian contexts are found in \cite[Prop.~6.2]{chmy}, \cite[Thm.~2.1]{bscv}, \cite[Thm.~5.3]{rr2} and \cite[Lem.~3.1(iv)]{r2}. Our approach requires approximation since the calibrated forms are not defined along the circles $C_0$ and $L$. Moreover, by analyzing the behaviour of minimizers near $C_0$, we prove that the half-spheres $\s_0^+(t)=\var_t(\s_0^+)$ are the unique area-minimizing $C^2$ surfaces with fixed boundary $C_0$.

\begin{theorem}
\label{th:main}
If $\Sg\subset\sph^3$ is a compact and orientable $C^1$ surface with $\ptl\Sg=C_0$, then $A(\Sg)\geq A(\s_0^+)$. Moreover, if equality holds and $\Sg$ is $C^2$, then $\Sg=\s_0^+(t)$ for some $t\in[0,2\pi)$.
\end{theorem}

\begin{proof}
Recall the notation $\mathcal{W}_0=\sph^3-C_0$. By Proposition~\ref{prop:foliation} (iv) there is a unit horizontal $C^\infty$ vector field $X^+_0$ on $\mathcal{W}_0-L$ such that $\divv X^+_0=0$ and the restriction of $X^+_0$ to $\s_0^+-(C_0\cup\{e\})$ equals the horizontal Gauss map $(\nu_0)_h$. The idea for the proof is to apply the divergence theorem to $X^+_0$ over the set bounded by $\Sg$ and $\s_0^+$. We have the difficulty that $X^+_0$ is defined on $\mathcal{W}_0-L$, whereas $\Sg$ is any surface in $\stres$ with $\ptl\Sg=C_0$. This is solved by means of an approximation argument. 

Starting from a sequence of logarithmic cut-off functions we can find, for any $\eps>0$ small enough, a $C^\infty$ function $\xi_\eps:[0,\pi/2]\to[0,1]$ satisfying:
\begin{itemize}
\item[(i)] $\xi_\eps=0$ in $[0,\eps^2]\cup [\frac{\pi}{2}-\eps^2,\frac{\pi}{2}]$,
\item[(ii)] $\xi_\eps=1$ in $[\eps,\frac{\pi}{2}-\eps]$,
\item[(iii)] $|\xi_\eps'|\leq 2/\eps$ in $[0,\frac{\pi}{2}]$.
\end{itemize}
Consider the rotationally invariant function $f_\eps:\stres\to[0,1]$ given by $f_\eps(p):=\xi_\eps(d_L(p))$. Since $d_L$ is $C^\infty$ on $\mathcal{W}_0-L$ and its gradient in $(\stres,g)$ verifies $|\nabla d_L|=1$, we get:
\begin{itemize}
\item[(a)] $f_\eps\in C^\infty(\stres)$ with $\text{supp}(f_\eps)\sub\mathcal{W}_0-L$,
\item[(b)] $\lim_{\eps\to 0}f_\eps(p)=1$, for any $p\in\mathcal{W}_0-L$,
\item[(c)] $\nabla f_\eps=0$ except for the compact region
\[
M_\eps:=\{p\in\stres\,;\,\eps^2\leq d_L(p)\leq\eps\}\cup\{p\in\stres\,;\,\frac{\pi}{2}-\eps\leq d_L(p)\leq\frac{\pi}{2}-\eps^2\},
\]
where we have $|\nabla f_\eps|\leq 2/\eps$.
\end{itemize}

For any $\eps>0$ small enough we define the vector field $Y_\eps:=f_\eps\,X^+_0$. Clearly $Y_\eps$ is a $C^\infty$ vector field on $\stres$ with $\text{supp}(Y_\eps)\sub\mathcal{W}_0-L$. Let us see that
\begin{equation}
\label{eq:main2}
\lim_{\eps\to 0}\int_{\Om}\divv Y_\eps\,dv=0,
\end{equation}
where $\Om\subeq\sph^3$ and $dv$ stands for the oriented volume form in $(\stres,g)$. To this aim, note that
\[
\divv Y_\eps=f_\eps\,\divv X^+_0+\escpr{\nabla f_\eps,X^+_0}=\escpr{\nabla f_\eps,X^+_0},
\]
since $\divv X^+_0=0$ and $\text{supp}(f_\eps)\sub\mathcal{W}_0-L$. So, we obtain
\[
\left|\int_\Om\divv Y_\eps\,dv\right|\leq\int_\Om|\escpr{\nabla f_\eps,X^+_0}|\,dv\leq\int_\Om|\nabla f_\eps|\,dv\leq\frac{2\,V(M_\eps)}{\eps}.
\] 
On the other hand, by taking the cylindrical coordinates in $\stres$ defined by equality \eqref{eq:cylindrical}, it is easy to check that the volume in $(\stres,g)$ of the compact set $\{p\in\stres\,;\,x\leq d_L(p)\leq y\}$ equals $2\pi^2\,(\sin^2y-\sin^2x)$. Hence, we deduce that
\[
\left|\int_\Om\divv Y_\eps\,dv\right|\leq 4\pi^2\,\frac{\cos(2\eps^2)-\cos(2\eps)}{\eps},
\]
so that \eqref{eq:main2} holds from L'H\^opital's rule by letting $\eps\to 0$.

Now, take a compact and orientable $C^1$ surface $\Sg$ in $\stres$ with $\ptl\Sg=C_0=\ptl\s_0^+$ and unit normal $N$. Since the chain $\Sg-\s_0^+$ is a $2$-cycle in $\stres$, then there is a $3$-chain $\Om$ in $\stres$ such that $\Sg-\s_0^+=\ptl\Om$. In $\stres$ we define the family of differential $2$-forms
\[
\zeta^\eps_p(u,w):=(dv)_p(u,w,Y_\eps(p))=f_\eps(p)\,\text{det}(u,w,X^+_0(p)),
\]
for any $p\in\stres$ and any $u,w\in T_p\stres$. It is clear that $\zeta_p^\eps(u,w)\to\zeta_p(u,w):=\text{det}(u,w,X^+_0(p))$ when $\eps\to 0$, for any $p\in\mathcal{W}_0-L$ and $u,w\in T_p\stres$. By applying Stokes theorem
\[
\int_\Om\divv Y_\eps\,dv=\int_\Om d\zeta^\eps=\int_\Sg\zeta^\eps-\int_{\s_0^+}\zeta^\eps.
\]
By using \eqref{eq:main2} and the dominated convergence theorem, we conclude that
\[
\int_{\s_0^+-(C_0\cup\{e\})}\zeta=\int_{\Sg-(C_0\cup L)}\zeta.
\]
On the other hand, from the Cauchy-Schwarz inequality we get 
\begin{equation}
\label{eq:defe}
\zeta_p(u,w)\leq \mnh(p),
\end{equation}
for any $p\in\Sg-(C_0\cup L)$ and any orthonormal basis $\{u,w\}$ in $T_p\Sg$. Moreover, we have
\[
\zeta_p(u,w)=|(N_0)_h|(p),
\] 
whenever $p\in\s_0^+-(C_0\cup\{e\})$ and $\{u,w\}$ is a positive orthonormal basis in $T_p\s_0^+$. Since $C_0\cup L$ does not contribute to the area of $\Sg$ and $\s_0^+$, we conclude by \eqref{eq:area} that
\[
A(\s_0^+)=\int_{\s_0^+}|(N_0)_h|\,da\leq\int_\Sg\mnh\,da=A(\Sg),
\]
which is the desired comparison.

Finally suppose that $A(\Sg)=A(\s_0^+)$, so that $\Sg$ solves the Plateau problem with fixed boundary $C_0$. This implies that $\Sg$ is connected; otherwise, the component of $\Sg$ containing $C_0$ would be a better competitor. Note that $C_0\sub\Sg-\Sg_0$ because the tangent lines along $C_0$ are all vertical. Let $\Sg'$ be the component of $\Sg-(\Sg_0\cup L)$ containing $C_0$. From equality $A(\Sg)=A(\s_0^+)$ and equation~\eqref{eq:defe} we deduce that $\nu_h=X^+_0$ or $\nu_h=-X^+_0$ on $\Sg'-C_0$. Hence the vector fields $\nu_h$ and $Z$ are $C^\infty$ on $\Sg'-C_0$. For a point $p\in\Sg'-C_0$ consider the unique half-sphere $\s^+_0(t)$ with $p\in\s_0^+(t)$, and denote by $\alpha_p$ and $\beta_p$ characteristic curves of $\Sg'$ and $\s^+_0(t)$ through $p$. By reversing $\alpha_p$ if necessary, and using that $(\nu_t)_h=X^+_0$ on $\s_0^+(t)-(C_0\cup L)$, we see that both curves are integral curves of the $C^\infty$ vector field $J(X^+_0)$ defined on $\mathcal{W}_0-L$. Thus $\alpha_p(\pm s)=\beta_p(s)$ around the origin, and so $\alpha_p$ is a horizontal geodesic in $(\stres,g)$. On the other hand, if $p\in C_0$ and $\alpha$ is a characteristic curve of $\Sg'$ with $\alpha(0)=p$, then there is $\delta>0$ such that $\alpha((0,\delta))\subset\Sg'-C_0$. From the previous argument the restriction $\alpha_{|(0,\delta)}$ is a horizontal geodesic in $(\stres,g)$. By continuity this geodesic extends to $0$ as a characteristic curve of $\Sg'$. All this shows that, locally around $C_0$, the surface $\Sg$ is union of horizontal geodesic segments in $(\stres,g)$ leaving orthogonally from $C_0$. Now the proof ends by invoking Proposition~\ref{prop:uniqueplateau} below.
\end{proof}

\begin{proposition}
\label{prop:uniqueplateau}
Let $\Sg\sub\stres$ be a compact, connected, orientable $C^2$ surface with $\ptl\Sg=C_0$ and containing geodesic segments in $(\stres,g)$ leaving orthogonally from $C_0$. If $\Sg\neq\s^+_0(t)$ for any $t\in[0,2\pi)$, then $A(\Sg)>A(\s^+_0)$.
\end{proposition}

\begin{proof}
We parameterize $C_0$ by $\Gamma(\eps):=(0,0,\cos(\eps),\sin(\eps))$ with $\eps\in[0,2\pi]$. From the hypotheses we can find $\delta>0$ such that $\ga_\eps(s)\in\Sg$ for any $(\eps,s)\in[0,2\pi]\times[0,\delta)$. Here $\ga_\eps:[0,2\pi]\to\stres$ is the great circle of $\stres$ described by
\begin{equation}
\label{eq:great}
\gamma_\eps(s):=\cos(s)\,\Gamma(\eps)+\sin(s)\,U(\eps),
\end{equation}
where the initial velocity $U(\eps):=\dot{\ga}_\eps(0)$ is a unit horizontal vector at $\Gamma(\eps)$. Hence, there is a function $\sg(\eps)$ with $\eps\in[0,2\pi]$ such that $U(\eps)=\cos\sigma(\eps)\,E_1(\eps)+\sin\sigma(\eps)\,E_2(\eps)$, where $E_i(\eps):=E_i(\Gamma(\eps))$ for any $i=1,2$. 

We define the map $F(\eps,s):=\ga_\eps(s)$ for $(\eps,s)\in[0,2\pi]\times[0,\pi]$. Observe that $F(\eps,s)\in C_0$ if and only if $s\in\{0,\pi\}$, whereas $F(\eps,s)\in L$ if and only if $s=\pi/2$. Moreover, $F$ is injective on the domains $[0,2\pi)\times[0,\pi/2)$ and $[0,2\pi)\times(\pi/2,\pi]$. The set $\{F(\eps,s)\,;\,(\eps,s)\in[0,2\pi]\times[0,\pi]\}$ corresponds to the case $\kappa=1$ and $\la=0$ of a more general construction studied by the authors~\cite[Sect.~4]{hr3} when the starting curve is the vertical axis $L$ instead of $C_0$. Since $C_0$ is also vertical, our analysis there extends to the present situation. In particular the angle $\sg(\eps)$, that determines $U(\eps)$ modulo $2\pi$, can be chosen as a $C^2$ function. 

Next, we gather some computations in the proof of \cite[Thm.~4.1]{hr3}. The coordinate vector fields associated to $F$ are given by $(\ptl F/\ptl\eps)(\eps,s):=V_\eps(s)$ and $(\ptl F/\ptl s)(\eps,s)=\dot{\ga}_\eps(s)$, where
\[
V_\eps(s)=\frac{v'_\eps(s)}{2}\,J(\dot{\ga}_\eps(s))+v_\eps(s)\,T_{\ga_\eps(s)},
\] 
and
\begin{equation}
\label{eq:plateau0}
v_\eps(s):=\frac{\sg'(\eps)-2}{2}\,\big(1-\cos(2s)\big)+1.
\end{equation}
Thus, the map $F$ fails to be an immersion only at the points $(\eps,\pi/2)$ such that $\sg'(\eps)=1$. Along the other points the area element satisfies
\begin{equation}
\label{eq:plateau1}
da=|\text{Jac}\,F|(\eps,s)\,d\eps\,ds=|V_\eps(s)|\,d\eps\,ds=\frac{\sqrt{4v_\eps(s)^2+v_\eps'(s)^2}}{2}\,d\eps\,ds,
\end{equation}
because $|\dot{\ga}_\eps(s)|=1$ and $\escpr{V_\eps(s),\dot{\ga}_\eps(s)}=0$. A pair $(\eps,s)$ for which $F$ is an immersion produces a singular point if and only if $\sg'(\eps)<1$ and $s\in\{s(\eps),\pi-s(\eps)\}$, where
\begin{equation}
\label{eq:plateau2}
s(\eps):=\frac{1}{2}\arccos\left(\frac{\sg'(\eps)}{\sg'(\eps)-2}\right).
\end{equation}
On the other hand, the map
\begin{equation}
\label{eq:plateau3}
N(\eps,s):=\frac{-2v_\eps(s)\,J(\dot{\ga}_\eps(s))+v_\eps'(s)\,T_{\ga_\eps(s)}}{\sqrt{4v_\eps(s)^2+v_\eps'(s)^2}}
\end{equation}
provides a unit normal when $F$ is an immersion. From here, we deduce
\begin{equation}
\label{eq:plateau4}
\mnh(\eps,s)=\frac{\pm 2\,v_\eps(s)}{\sqrt{4v_\eps(s)^2+v_\eps'(s)^2}},
\end{equation}
with positive sign if and only if $\sg'(\eps)\geq 1$ or $\sg'(\eps)<1$ and $s\in[0,s(\eps)]\cup[\pi-s(\eps),\pi]$. In such cases, we  compute $Z(\eps,s)$ from \eqref{eq:plateau3} to get $\dot{\ga}_\eps(s)=Z(\eps,s)$. As a consequence, we can choose the orientation of $\Sg$ so that any $\ga_\eps(s)$ with $s\in[0,\delta)$ is a characteristic curve of $\Sg$. As $\Sg$ is compact this property holds until $\ga_\eps$ meets a singular point.

Now, let us prove that $\Sg$ cannot contain singular curves. On the contrary, we would have $\sg'<1$ in an open interval $I\subset[0,2\pi]$. Observe that the singular curve $c(\eps):=F(\eps,s(\eps))$ with $\eps\in I$ does not touch the set $C_0\cup L$. Since $\Sg-C_0$ is an area-stationary $C^2$ surface we know from \cite[Thm.~4.17]{rr2}, see also \cite[Thm.~4.5]{hr1}, that the characteristic curves $\ga_\eps(s)$ with $s\in(0,s(\eps))$ meet orthogonally $c(\eps)$. By having in mind that
\[
\dot{c}(\eps)=V_\eps(s(\eps))+s'(\eps)\,\dot{\ga}_\eps(s(\eps)), \quad \eps\in I,
\]
this is equivalent to that $s'(\eps)=0$ for any $\eps\in I$. From equation~\eqref{eq:plateau2} there are constants  $s_0\in (0,\pi/2)$ and $\sg_0'<1$ such that $s(\eps)=s_0$ and $\sg'(\eps)=\sg'_0$, for any $\eps\in[0,2\pi]$. Note that $N(\eps,s_0)=-T_{c(\eps)}$ by \eqref{eq:plateau3}. Next, according with the results in \cite[Prop.~3.5, Cor.~3.6]{chmy} (see also \cite[Thm.~4.3]{hr1}), that describe locally $\Sg-C_0$ around $c(\eps)$, we can continue the surface $\Sg$ beyond $c(\eps)$ by following the curves $\ga_\eps(s)$ for $s\geq s_0$ small enough. Indeed, from the discussion below \eqref{eq:plateau4} we obtain characteristic curves by reversing the orientation of $\ga_\eps(s)$ with $s>s_0$. By using that $\Sg$ is compact and that the curve $L$ (which we reach when $s=\pi/2$) is vertical, we can prolong the curves $\ga_\eps(s)$ inside $\Sg$ until we meet a new singular curve $d(\eps):=F(\eps,\pi-s_0)$. Observe that $N(\eps,\pi-s_0)=T_{d(\eps)}$ and so, $d(\eps)$ is disjoint from $c(\eps)$. Finally, we reason as above to conclude that $\ga_\eps(s)$ with $s>\pi-s_0$ is a characteristic curve of $\Sg$ until we come back to the starting circle $C_0$ at $s=\pi$. In this way we would produce an open neighborhood of $C_0$ in $\Sg$ by matching two surfaces where $\escpr{N,T}<0$ and $\escpr{N,T}>0$, respectively. This contradicts that $C_0=\ptl\Sg$ because $\Sg$ is embedded.

Now, we are ready to prove the claim. From the fact that $\Sg$ does not contain singular curves we infer that $\sg'\geq 1$ in $[0,2\pi]$. Hence the set $\widetilde{\Sg}:=F([0,2\pi)\times[0,\pi/2))$ is an immersed surface with $\widetilde{\Sg}_0=\emptyset$ and $\widetilde{\Sg}\cap L=\emptyset$. Note that $\widetilde{\Sg}\subseteq\Sg$ because $\Sg$ is compact. From equations \eqref{eq:plateau4}, \eqref{eq:plateau1} and \eqref{eq:plateau0}, we get 
\begin{equation}
\label{eq:plateau5}
\begin{split}
A(\Sg)\geq A(\widetilde{\Sg})&=\int_{\widetilde{\Sg}}\mnh\,da=\int_{[0,2\pi)\times[0,\pi/2)}\mnh(\eps,s)\,|\text{Jac}\,F|(\eps,s)\,d\eps\,ds
\\
&=\int_{[0,2\pi)\times[0,\pi/2)}v_\eps(s)\,d\eps\,ds=\frac{\pi}{4}\,\int_0^{2\pi}\sg'(\eps)\,d\eps\geq\frac{\pi^2}{2}.
\end{split}
\end{equation}
If equality holds then $\sg'=1$ in $[0,2\pi]$, so that $\sg(\eps)=\eps+\eps_0$ for some $\eps_0\in[0,2\pi)$. Thus $U(\eps)=-(\cos(\eps_0),\sin(\eps_0),0,0)$, and we would deduce from \eqref{eq:great} that $\widetilde{\Sg}\subseteq\s_0(\eps_0)$, where $\s_0(\eps_0)=\var_{\eps_0}(\s_0)$ is the totally geodesic $2$-sphere defined by intersecting $\stres$ with the hyperplane of equation $\sin(\eps_0)\,x_1-\cos(\eps_0)\,y_1=0$. More precisely, we have $\widetilde{\Sg}=\s^+_0(\eps_0\pm\pi)-\{\text{pole}\}$ (with positive sign if and only if $\eps_0\in[0,\pi)$). As $\Sg$ is compact and connected with $\ptl\Sg=C_0$ this would give $\Sg=\s^+_0(t)$ for some $t\in[0,2\pi)$, which contradicts the hypotheses. Therefore, inequality \eqref{eq:plateau5} reads
\[
A(\Sg)>\frac{\pi^2}{2}=A(\s^+_0(t))=A(\s^+_0),
\]
and the proof is completed.
\end{proof}

As a direct consequence of Theorem~\ref{th:main} we infer a sharp lower bound for the area of compact surfaces without boundary and containing $C_0$.

\begin{corollary}
\label{cor:sphere}
If $\Sg\subset\stres$ is a compact $C^1$ surface with empty boundary such that $C_0\subset\Sg$ and $\Sg-C_0$ is disconnected, then $A(\Sg)\geq\pi^2$. Moreover, if $A(\Sg)=\pi^2$ and $\Sg$ is $C^2$, then $\Sg=\var_t(\s_0)$ for some $t\in[0,2\pi)$. In particular $\s_0$ uniquely minimizes the area, up to vertical translations, among all the $C^2$ topological $2$-spheres of $\stres$ containing $C_0$.
\end{corollary}

\section{Isoperimetric property of the spheres $\s_\la$}
\label{sec:isoperimetric}

In this section we show that the spherical surfaces $\s_\la$ with $\la>0$ are solutions to the isoperimetric problem in $(\stres,g_h)$ among sets inside a vertical solid tube and containing a horizontal section of the tube. This provides a counterpart in $(\stres,g_h)$ of a theorem by Ritor\'e for the sub-Riemannian Heisenberg group~~\cite{ritore-calibrations}. Indeed, we will employ the foliations constructed in Section~\ref{sec:foliation} for adapting the arguments in \cite[Thm.~3.1]{ritore-calibrations} to the present situation.

\begin{theorem}
\label{th:main2}
Let $\Om\subset\stres$ be a set with $C^1$ boundary $\Sg$ and volume $V(\Om)\leq V(\stres)/2$. Suppose that there is $\la>0$ such that $\Om$ is contained in the tube $\mathcal{W}_\la$ defined in \eqref{eq:wla} and $\overline{\Om}\cap\sph^2_\la=D_\la$. Then, we have
\[
A(\Sg)\geq A(\s_\mu),
\]
where $\s_\mu$ is the spherical surface in $(\stres,g_h)$ with $V(\be_\mu)=V(\Om)$. Moreover, if equality holds, then $\overline{\Om}=\be_\mu$.
\end{theorem}

\begin{proof}
For simplicity we will denote $\Om^+:=\Om_\la^+$, $\Om^-:=\Om_\la^-$, $\Sg^+:=\Sg_\la^+$ and $\Sg^-:=\Sg_\la^-$, see equation~\eqref{eq:halves}. We first derive the minimization property
\begin{equation}
\label{eq:x1}
A(\Sg)-2\la\,V(\Om)\geq A(\s_\la)-2\la\,V(\be_\la).
\end{equation}
For that we will apply the divergence theorem over $\Om^+$ and $\Om^-$ to the vector fields $X^+_\la$ and $X^-_\la$ constructed in Proposition~\ref{prop:foliation} (iv). This requires approximation since the set $\Sg\cap (D_\la-C_\la)$ could be non-empty and $X^+_\la$, $X^-_\la$ are defined on $\mathcal{W}_\la-L$ whereas $\Om\cap L\neq\emptyset$.  

As in the proof of Theorem~\ref{th:main} we can find a sequence of smooth functions $f_\eps:\mathcal{W}_\la\to [0,1]$, for $\eps>0$ small enough, such that:
\begin{itemize}
\item[(a)] $\text{supp}(f_\eps)\subeq\mathcal{W}_\la-L$,
\item[(b)] $\lim_{\eps\to 0}f_\eps(p)=1$, for any $p\in W_\la-L$,
\item[(c)] $\nabla f_\eps=0$ except for the compact region $
K_\eps:=\{p\in\stres\,;\,\eps^2\leq d_L(p)\leq\eps\}$, where we have $|\nabla f_\eps|\leq 2/\eps$.
\end{itemize}

Consider the foliation of $\mathcal{W}_\la$ by vertical translations of $\s_\la^+$. The associated vector field $X^+_\la$ is a unit horizontal $C^\infty$ vector field on $\mathcal{W}_\la-L$, whose restriction to $\s_\la^+-\{e\}$ coincides with the horizontal Gauss map $(\nu_0)_h$, and verifies $\divv X^+_\la=-2\la$ on $\mathcal{W}_\la-L$. We define $Y_\eps:=f_\eps\,X^+_\la$, which is a $C^\infty$ vector field on $\mathcal{W}_\la$ with $\text{supp}(Y_\eps)\sub\mathcal{W}_\la-L$. Since $\divv Y_\eps=-2\la\,f_\eps+\escpr{\nabla f_\eps,X^+_\la}$, we can proceed as in the proof of \eqref{eq:main2} with the help of the dominated convergence theorem to deduce
\begin{equation}
\label{eq:x2}
\lim_{\eps\to 0}\int_E\divv Y_\eps\,dv=-2\la\,V(E),\quad\text{ for any set }E\subeq\mathcal{W}_\la.
\end{equation}

Now, we take a set $\Om$ as in the statement. We will approximate $\Om$ by sets $\Om_t$ in similar conditions and satisfying that $\ptl\Om_t\cap D_\la=C_\la$. We remark that this procedure is unnecessary when $\Sg\cap D_\la=C_\la$, as happens for instance if $\Om=\be_\la$.

For any $t\geq0$ we consider the set $\var_t(\Om^+)$, where $\var_t$ is the vertical translation in \eqref{eq:verttrans}. Since $\overline{\Om}\cap\sph^2_\la=D_\la$ and $\mathcal{W}_\la$ is invariant under $\var_t$, there is $\delta>0$ such that $\var_t(\Om^+)\sub\mathcal{W}_\la$ and $\overline{\var_t(\Om^+)}\cap\sph^2_\la=\emptyset$ for any $t\in (0,\delta]$. We denote by $\Om^+_t$ the union of $\var_t(\Om^+)$ with the small vertical tube $\bigcup_{s\in[0,t]}\var_s(D_\la)$. The boundary $\ptl\Om^+_t$ is the union of $\var_t(\Sg^+)$ together with $D_\la$ and a small piece $\mathcal{C}_t$ of the Clifford torus $\ptl\mathcal{W}_\la$. Hence, the divergence theorem applied to the vector field $Y_\eps$ over $\Om^+_t$ yields
\[
\int_{\Om^+_t}\divv Y_\eps\,dv=-\int_{D_\la}f_\eps\,\escpr{X^+_\la,N_{D_\la}}\,da-\int_{\var_t(\Sg^+)}f_\eps\,\escpr{X_\la^+,N_t}\,da-\int_{\mathcal{C}_\la}f_\eps\,\escpr{X^+_\la,\eta}\,da,
\]
where $N_{D_\la}$, $N_t$ and $\eta$ are the corresponding unit normals in $(\stres,g)$ pointing into $\Om^+_t$. Taking limits when $\eps\to 0$ we get from \eqref{eq:x2} and the dominated convergence theorem
\begin{align*}
-2\la\,V(\Om^+_t)&=-\int_{D_\la}\escpr{X^+_\la,N_{D_\la}}\,da-\int_{\var_t(\Sg^+)}\escpr{X_\la^+,N_t}\,da-\int_{\mathcal{C}_t}\escpr{X^+_\la,\eta}\,da
\\
&=-\int_{D_\la}\escpr{X^+_\la,N_{D_\la}}\,da-\int_{\Sg^+}\escpr{X_\la^+,N}\,da-\int_{\mathcal{C}_t}\escpr{X^+_\la,\eta}\,da,
\end{align*}
where $N$ is the normal on $\Sg$ pointing into $\Om$, and we have employed the change of variables formula together with equalities $X^+_\la\circ\var_t=\var_t\circ X^+_\la$ and $N_t\circ\var_t=\var_t\circ N$. By passing to the limit when $t\to 0$ and having in mind that $\{V(\Om^+_t)\}\to V(\Om^+)$ and $\{A(\mathcal{C}_t)\}\to 0$, we infer
\begin{equation}
\label{eq:x21}
\begin{split}
-2\la\,V(\Om^+)&=-\int_{D_\la}\escpr{X^+_\la,N_{D_\la}}\,da-\int_{\Sg^+}\escpr{X^+_\la,N}\,da
\\
&\geq -\int_{D_\la}\escpr{X^+_\la,N_{D_\la}}\,da-A(\Sg^+),
\end{split}
\end{equation}
where we have used the Cauchy-Schwarz inequality and that $X^+_\la$ is a unit horizontal vector field. In case $\Om=\be_\la$ we get equality above since $X^+_\la=(\nu_0)_h$ on $\s_\la^+-\{e\}$. This means that
\[
-2\la\,V(\be_\la^+)=-\int_{D_\la}\escpr{X^+_\la,N_{D_\la}}\,da-A(\s_\la^+).
\]
From the last two equations we derive the inequality
\begin{equation}
\label{eq:x3}
A(\Sg^+)-2\la\,V(\Om^+)\geq A(\s_\la^+)-2\la\,V(\be_\la^+).
\end{equation}
If we reproduce the previous arguments with the set $\Om^-$ and the vector field $X_\la^-$ associated to the foliation of $\mathcal{W}_\la$ by vertical translations of $\s_\la^-$, then we conclude that
\begin{equation}
\label{eq:x4}
A(\Sg^-)-2\la\,V(\Om^-)\geq A(\s_\la^-)-2\la\,V(\be_\la^-).
\end{equation}
Thus \eqref{eq:x1} follows by adding \eqref{eq:x3} and \eqref{eq:x4} since $D_\la$ has no contribution to $V(\Om)$ and $V(\be_\la)$.

Next, we define the function $\xi:[0,+\infty)\to\rr$ by $\xi(\alpha):=A(\s_\alpha)+2\alpha\,(V(\Om)-V(\be_\alpha))$. By taking derivatives $'$ with respect to $\alpha$, we infer that
\begin{align*} 
\xi'(\alpha)&=A'(\s_\alpha)-2\alpha\,V'(\be_\alpha)+2\,\big(V(\Om)-V(\be_\alpha)\big)=2\,\big(V(\Om)-V(\be_\alpha)\big),
\\
\xi''(\alpha)&=-2\,V'(\be_\alpha),
\end{align*}
where in the first equality we have used \eqref{eq:eso} and that the spheres $\s_\alpha$ are volume-preserving area-stationary. As $V(\be_\alpha)$ is decreasing for $\alpha\in(0,+\infty)$ by Proposition~\ref{prop:spheres} (vii), we conclude that $\xi$ is strictly convex and attains its minimum at the unique value $\mu\in [0,+\infty)$ for which $V(\be_\mu)=V(\Om)$ (this value exists because $V(\Om)\leq V(\stres)/2$). From \eqref{eq:x1} we get
\[
A(\Sg)\geq\xi(\la)\geq\xi(\mu)=A(\s_\mu),
\]
which is the desired area comparison.

Finally, suppose that $A(\Sg)=A(\s_\mu)$. Then we have $\xi(\la)=\xi(\mu)$, so that $\mu=\la>0$. Moreover, we also obtain equalities in \eqref{eq:x3} and \eqref{eq:x4}, which entails by \eqref{eq:x21} that
\begin{equation}
\label{eq:x5}
\begin{split}
\nu_h&=X^+_\mu \ \text{ on } \Sg^+-(\Sg_0\cup L),
\\
\nu_h&=X^-_\mu \ \text{ on } \Sg^--(\Sg_0\cup L).
\end{split}
\end{equation}
Recall that the same equalities hold in $\s^+_\mu-L$ and $\s^-_\mu-L$ by the definition of $X^+_\mu$ and $X^-_\mu$. 

Let us see that $\overline{\Om}=\be_\mu$. It is clear that the circle $C_\mu=\ptl D_\mu$ is contained in $\Sg-L$ and $\s_\mu-L$. The surfaces $\Sg$ and $\s_\mu$ are tangent along $C_\mu$ because $\overline{\Om},\mathcal{B}_\mu\subeq\mathcal{W}_\mu$ and $C_\mu\subset\ptl\mathcal{W}_\mu$. As the singular set $(\s_\mu)_0$ equals $\{e,e^\mu\}\sub L$, it follows that $C_\mu\subset\Sg-\Sg_0$ and $C_\mu\subset\s_\mu-(\s_\mu)_0$. For any $p\in C_\mu$ we denote by $\beta_p$ the maximal characteristic curve of $\s_\mu$ passing through $p$. By definition of $\s_\mu$ and Proposition~\ref{prop:spheres} (iv) the trace of $\beta_p$ equals the trace of a geodesic $\gamma_\theta:(0,\pi/\sqrt{1+\mu^2})\to\s_\mu$ defined in \eqref{eq:geocoor}. Let $\alpha_p$ be a characteristic curve of $\Sg-L$ through $p$. By \eqref{eq:x5} we get $\alpha_p'(0)=\beta_p'(0)$, which is a tangent vector transversal to $\sph^2_\mu$. Hence, we can find $\delta_p>0$ small enough so that $\alpha_p([0,\delta_p))\sub\Sg^--(\Sg_0\cup L)$ and $\alpha_p((-\delta_p,0])\sub \Sg^+-(\Sg_0\cup L)$. Again from \eqref{eq:x5} we see that the restrictions of $\beta_p$ and $\alpha_p$ to $[0,\delta_p)$ (resp. $(-\delta_p,0]$) are integral curves through $p$ of the $C^\infty$ vector field  $J(X_\mu^-)$ (resp. $J(X^+_\mu)$) defined on $\mathcal{W}_\mu-L$. This implies that $\alpha_p=\beta_p$ in $[-\delta_p,\delta_p]$. Moving $p$ along $C_\mu$ we deduce that $\Sg-L$ contains a region $\mathcal{R}\sub\s_\mu-L$ given by the union of the curves $\beta_p(s)$ with $s\in [-\delta,\delta]$. Since $\Sg$ is compact and tangent to $\s_\mu$ over $\mathcal{R}$, the curves $\beta_p$ can be extended as characteristic curves of $\Sg-L$ until they meet $L$ or $\Sg_0$. This shows that $\s_\mu-L\subeq\Sg-L$. As a consequence $\s_\mu\subeq\Sg$ because $\Sg$ is a closed subset of $\stres$. From the fact that $A(\Sg)=A(\s_\mu)$ we obtain $\Sg=\s_\mu$. By having in mind that $D_\mu\sub\overline{\Om}\cap\be_\mu$, we conclude that $\overline{\Om}=\be_\mu$. This proves the theorem.
\end{proof}

\begin{remarks}
1. If we remove the hypothesis $V(\Om)\leq V(\stres)/2$ then we can use the inequality $V(\Om)\geq v(\Om):=\min\{V(\Om),V(\stres-\Om)\}$ into \eqref{eq:x1}, and define $\xi(\alpha):=A(\s_\alpha)+2\alpha\,(v(\Om)-V(\be_\alpha))$, to deduce that $A(\Sg)\geq\xi(\la)\geq\xi(\mu)=A(\s_\mu)$, where $\s_\mu$ is the spherical surface with $V(\be_\mu)=v(\Om)$. Moreover, if equality holds, then $V(\Om)\leq V(\stres)/2$ and $\overline{\Om}=\be_\mu$.

2. The area comparison is still valid when $\la=0$. In this case we only suppose that $\Om\subseteq\stres$ with $V(\Om)\leq V(\stres)/2$ and $\overline{\Om}\cap\sph^2_0=D_0$. Observe that the inequality $A(\Sg)\geq A(\s_0)$ in \eqref{eq:x1} comes from Theorem~\ref{th:main} since $\ptl\Sg^+=\ptl\Sg^-=C_0$. The rest of the proof continues without changes. If equality holds then $A(\Sg)=A(\s_0)$. When $\Sg$ is a $C^2$ surface, the condition $\overline{\Om}\cap\sph^2_0=D_0$ together with Theorem~\ref{th:main} lead us to $\overline{\Om}=\be_0$.

3. The arguments can be slightly modified to show that the isoperimetric inequality in the statement still holds when $\Om$ is a finite perimeter set in $(\stres,g_h)$ under the same hypotheses. The characterization of equality cases in this context is more delicate and would require a regularity result as the one of Monti and Vittone~\cite[Thm.~1.2]{monti-vittone} employed by Ritor\'e~\cite[Thm.~3.1]{ritore-calibrations} in the Heisenberg group. 
\end{remarks}

\providecommand{\bysame}{\leavevmode\hbox to3em{\hrulefill}\thinspace}
\providecommand{\MR}{\relax\ifhmode\unskip\space\fi MR }
\providecommand{\MRhref}[2]{%
  \href{http://www.ams.org/mathscinet-getitem?mr=#1}{#2}
}
\providecommand{\href}[2]{#2}


\begin{thebibliography}{10}

\bibitem{bscv}
V.~Barone~Adesi, F.~Serra~Cassano, and D.~Vittone, \emph{The {B}ernstein
  problem for intrinsic graphs in {H}eisenberg groups and calibrations}, Calc.
  Var. Partial Differential Equations \textbf{30} (2007), no.~1, 17--49.
  \MR{MR2333095}

\bibitem{MR2583494}
L.~Capogna, G.~Citti, and M.~Manfredini, \emph{Regularity of non-characteristic
  minimal graphs in the {H}eisenberg group {$\Bbb H^1$}}, Indiana Univ. Math.
  J. \textbf{58} (2009), no.~5, 2115--2160. \MR{2583494 (2010j:58032)}

\bibitem{survey}
L.~Capogna, D.~Danielli, S.~D. Pauls, and J.~T. Tyson, \emph{An introduction to
  the {H}eisenberg group and the sub-{R}iemannian isoperimetric problem},
  Progress in Mathematics, vol. 259, Birkh\"auser Verlag, Basel, 2007.
  \MR{MR2312336}

\bibitem{chanillo-yang}
S.~Chanillo and P.~Yang, \emph{Isoperimetric inequalities \& volume comparison
  theorems on {CR} manifolds}, Ann. Sc. Norm. Super. Pisa Cl. Sci. (5)
  \textbf{8} (2009), no.~2, 279--307. \MR{MR2548248}

\bibitem{cchy2}
J.-H. Cheng, H.-L. Chiu, J.-F. Hwang, and P.~Yang, \emph{Strong maximum
  principle for mean curvature operators on sub{R}iemannian manifolds}, Math.
  Ann. \textbf{372} (2018), no.~3-4, 1393--1435. \MR{3880302}

\bibitem{chmy}
J.-H. Cheng, J.-F. Hwang, A.~Malchiodi, and P.~Yang, \emph{Minimal surfaces in
  pseudohermitian geometry}, Ann. Sc. Norm. Super. Pisa Cl. Sci. (5) \textbf{4}
  (2005), no.~1, 129--177. \MR{MR2165405 (2006f:53008)}

\bibitem{chmy2}
\bysame, \emph{A {C}odazzi-like equation and the singular set for {$C^1$}
  smooth surfaces in the {H}eisenberg group}, J. Reine Angew. Math.
  \textbf{671} (2012), 131--198. \MR{2983199}

\bibitem{chy}
J.-H. Cheng, J.-F. Hwang, and P.~Yang, \emph{Existence and uniqueness for
  {$p$}-area minimizers in the {H}eisenberg group}, Math. Ann. \textbf{337}
  (2007), no.~2, 253--293. \MR{MR2262784}

\bibitem{hr1}
A.~Hurtado and C.~Rosales, \emph{Area-stationary surfaces inside the
  sub-{R}iemannian three-sphere}, Math. Ann. \textbf{340} (2008), no.~3,
  675--708. \MR{MR2358000 (2008i:53038)}

\bibitem{hr2}
\bysame, \emph{Existence, characterization and stability of {P}ansu spheres in
  sub-{R}iemannian 3-space forms}, Calc. Var. Partial Differential Equations
  \textbf{54} (2015), no.~3, 3183--3227. \MR{3412407}

\bibitem{hr3}
\bysame, \emph{Strongly stable surfaces in sub-{R}iemannian 3-space forms},
  Nonlinear Anal. \textbf{155} (2017), 115--139. \MR{3631745}

\bibitem{hr4}
\bysame, \emph{An instability criterion for volume-preserving area-stationary
  surfaces with singular curves in sub-{R}iemannian 3-space forms}, Calc. Var.
  Partial Differential Equations \textbf{59} (2020), no.~5, Paper No. 165, 34.
  \MR{4149341}

\bibitem{miranda-bv}
M.~Miranda, \emph{Functions of bounded variation on ``good'' metric spaces}, J.
  Math. Pures Appl. (9) \textbf{82} (2003), no.~8, 975--1004. \MR{2005202
  (2004k:46038)}

\bibitem{mscv}
R.~Monti, F.~Serra~Cassano, and D.~Vittone, \emph{A negative answer to the
  {B}ernstein problem for intrinsic graphs in the {H}eisenberg group}, Boll.
  Unione Mat. Ital. (9) \textbf{1} (2008), no.~3, 709--727. \MR{2455341}

\bibitem{monti-vittone}
R.~Monti and D.~Vittone, \emph{Sets with finite {$\Bbb H$}-perimeter and
  controlled normal}, Math. Z. \textbf{270} (2012), no.~1-2, 351--367.
  \MR{2875838}

\bibitem{pansu2}
P.~Pansu, \emph{Une in\'egalit\'e isop\'erim\'etrique sur le groupe de
  {H}eisenberg}, C. R. Acad. Sci. Paris S\'er. I Math. \textbf{295} (1982),
  no.~2, 127--130. \MR{MR676380 (85b:53044)}

\bibitem{pansu1}
\bysame, \emph{An isoperimetric inequality on the {H}eisenberg group}, Rend.
  Sem. Mat. Univ. Politec. Torino (1983), no.~Special Issue, 159--174 (1984),
  Conference on differential geometry on homogeneous spaces (Turin, 1983).
  \MR{MR829003 (87e:53070)}

\bibitem{pauls-regularity}
S.~D. Pauls, \emph{{$H$}-minimal graphs of low regularity in {$\Bbb H\sp 1$}},
  Comment. Math. Helv. \textbf{81} (2006), no.~2, 337--381. \MR{MR2225631
  (2007g:53032)}

\bibitem{pozuelo-ritore}
J~Pozuelo and M.~Ritor\'e, \emph{Pansu-{W}ulff shapes in $\mathbb{H}^1$}, to
  appear in Adv.~Calc.~Var, arXiv:2007.04683, July 2020.

\bibitem{santalo-sr}
D.~Prandi, L.~Rizzi, and M.~Seri, \emph{A sub-{R}iemannian {S}antal\'{o}
  formula with applications to isoperimetric inequalities and first {D}irichlet
  eigenvalue of hypoelliptic operators}, J. Differential Geom. \textbf{111}
  (2019), no.~2, 339--379. \MR{3909911}

\bibitem{r2}
M.~Ritor\'e, \emph{Examples of area-minimizing surfaces in the sub-{R}iemannian
  {H}eisenberg group $\mathbb{H}^1$ with low regularity}, Calc. Var. Partial
  Differential Equations \textbf{34} (2009), no.~2, 179--192. \MR{MR2448649
  (2009h:53062)}

\bibitem{ritore-libro}
\bysame, \emph{Geometric flows, isoperimetric inequalities and hyperbolic
  geometry}, Mean curvature flow and isoperimetric inequalities, Adv. Courses
  Math. CRM Barcelona, Birkh{\"a}user Verlag, Basel, 2010, pp.~45--113.
  \MR{2590632}

\bibitem{ritore-calibrations}
\bysame, \emph{A proof by calibration of an isoperimetric inequality in the
  {H}eisenberg group {${\Bbb H}^n$}}, Calc. Var. Partial Differential Equations
  \textbf{44} (2012), no.~1-2, 47--60. \MR{2898770}

\bibitem{rr2}
M.~Ritor\'e and C.~Rosales, \emph{Area-stationary surfaces in the {H}eisenberg
  group {$\Bbb H\sp 1$}}, Adv. Math. \textbf{219} (2008), no.~2, 633--671.
  \MR{MR2435652}

\end{thebibliography}
\end{document}